\documentclass[twoside,12pt]{amsart}
\usepackage{amsmath}
\usepackage{amsthm}
\usepackage{amsfonts}
\usepackage{latexsym}
\usepackage{float}
\usepackage{upref}
\restylefloat{figure}
\usepackage{mathptmx}

\makeatletter
\def\serieslogo@{}
\makeatother \makeatletter
\def\@setcopyright{}

\usepackage[dvips]{graphicx}

\def\bdm{\begin{displaymath}}
\def\edm{\end{displaymath}}
\def\be{\begin{equation}}
\def\ee{\end{equation}}
\newcommand{\arginf}[1]{\mathop{arg\,min}_{#1}}
\def\et{\eta}
\def\lam{\lambda}
\def\two{2} 
 
\def\i{1}  
\def\ii{2} 
\def\lami{{\lam_\i}}
\def\lamii{{\lam_\ii}}

\def\bra{[}
\def\ket{]}

\def\bbv{{\mathbf v}} 

\def\J{{\mathcal \vee}}
\def\JT{\J} 

\def\qmin{{v}_{min}}

\def\h{\bbv} 

\def\T{{\mathcal L}}  
\def\Tdual{{\T^{*}\!}}   
\def\P{{\mathcal P}}                       
\def\Pdual{{\P^{*}}}

\def\x{U}  
\def\u{{\mathbf u}}  
\def\ueps{\u^\epsilon}
\def\uepsa{u^\epsilon_1}
\def\uepsb{u^\epsilon_2}
\def\omeps{\omega^\epsilon}
\def\ff{{\mathbf f}}
\def\r{{r}}  
\def\rg{{g}}

\def\bg{{\mathbf g}}

\def\B{{\mathbb B}}  
\def\Bdual{{\B^{*}}}
\newcommand{\Bnorm}[1]{ {\|#1\|_{\B}}}    

\def\Leb{L}  
\def\Wdot{\dot{W}} 
\def\Hdot{{\dot{H}}} 
\def\BV{\Wdot^{1,1}}
\def\OPhi{{\Leb}^{p}}				
\def\OPhid{{\Leb}^{p}_{\#}}                   
\def\OPsi{{\Leb}^{p'}}				
\def\OPsid{{\Leb}^{p'}_{\#}}                  
\newcommand{\Lp}[1]{\|#1\|^p_{\Leb^p}} 
\newcommand{\dualnorm}[1]{ \|#1\|_{\Bdual}}    
\newcommand{\myphi}[1]{\varphi(#1)}   
\newcommand{\Lin}[1]{\Tdual\varphi(#1)} 
\def\div{div\,}
\def\curl{curl\,}

\def\RR{\mathbb R}
\def\TT{\mathbb T}  
\def\Om{\Omega}     
\def\BVbeta{\beta}  
\def\BVgamma{\gamma}
\def\BVbetap{\et} 
\def\two{2}   

\def\lesssim{\stackrel{{}_<}{{}_\sim}}
\def\geqsim{\stackrel{{}_>}{{}_\sim}}
\def\eqsim{{\Large \, \stackrel{{}_{\sim}}{{}_{=}}\, }}

\newtheorem{thm}{Theorem}[section]
\newtheorem{theorem*}{Theorem}
\newtheorem{lem}[thm]{Lemma}
\newtheorem{cor}[thm]{Corollary}

\newtheorem{assumption*}[theorem*]{Assumption}
\theoremstyle{definition}

\newtheorem{rem}{Remark}[section]

\numberwithin{equation}{section}
\numberwithin{thm}{section}
\numberwithin{rem}{section}
\numberwithin{defn}{section}

\begin{document}

\footnotetext{July 6, 2014} 

\title[Hierarchical construction of bounded solutions\\in critical regularity spaces]
{Hierarchical construction of bounded solutions\\in critical regularity spaces}  

\author{Eitan Tadmor}
\address{Department of Mathematics, Center of Scientific Computation And Mathematical Modeling (CSCAMM)
and Institute for Physical Science and Technology (IPST), University of Maryland, MD 20742.}
\curraddr{} \email{tadmor@cscamm.umd.edu}
\thanks{Research was supported  by NSF grants DMS10-08397, RNMS11-07444 (KI-Net) and   ONR grant \#N00014-1210318.} 

\subjclass{35C10, 35F05, 46A30, 49M27, 68U10}
\keywords{Closed range theorem, hierarchical solutions, duality, Sobolev estimates, divergence, curl, multiscale expansion, image processing, critical regularity.}
\setcounter{page}{1}

\begin{abstract}
We construct uniformly bounded solutions for the equations $\div\x=f$ and $\curl\x=\ff$ in the critical cases
$f\in \Leb^d_\#(\TT^d,\RR)$, and respectively, $\ff\in \Leb^3_\#(\TT^3,\RR^3)$. Criticality in this context, manifests itself by the lack of linear solution operator mapping $\Leb^d_\# \mapsto \Leb^\infty(\TT^d)$, \cite{BB03,BB07}. Thus, the intriguing aspect here is that although the problems are linear, the construction of their  solution is not. 

 Our constructions are special cases of  a general framework for solving linear equations of the form $\T\x=f$, where  $\T$ is a linear operator densely defined in Banach space $\B$ with a closed range in a (proper subspace) of Lebesgue space $\Leb^p_\#(\Om)$, and with an injective dual $\Tdual$. The solutions are realized  in terms of a
 multiscale \emph{hierarchical representation}, $\x=\sum_{j=1}^\infty\u_j$, interesting for its own sake. Here, the $\u_j$'s are constructed 
\emph{recursively} as minimizers of
${\u_{j+1} = \arginf{\u}\big\{\Bnorm{\u}+\lam_{j+1}\|\r_j-\T\u\|^p_{\Leb^p(\Om)}\big\}}$, where $\r_j:=f-\T(\sum^j_{k=\i}\u_k)$
are resolved in terms of  a dyadic sequence of scales $\lam_{j+1}:=\lami \two^j$ with large enough $\lam_1 \geqsim \|f\|_{\Leb^p}^{1-p}$. 
The  nonlinear aspect of this  construction is a  counterpart of the  fact that one cannot linearly solve $\T\x=f$ in critical spaces.
\end{abstract}

\maketitle
\vspace*{-0.8cm}
{\small \tableofcontents}

\markboth{{\sc Eitan Tadmor}}{{\sc Construction of hierarchical solutions in critical regularity spaces}}

\vspace*{-0.8cm}
\begin{verbatim}
      "Whenever you can settle a question by explicit construction, 
       be not satisfied with purely existential arguments."
\end{verbatim}

\medskip
\hspace{8.2cm}{Hermann Weyl, \cite[p. 326]{Weyl46}}

\section{Introduction and statement of the main results}

We begin with a prototype example for the class of equations alluded to in the title of the paper.
Let $\Leb^d_\#(\TT^d)$ denote the Lebesgue space of  periodic functions with zero mean over the $d$-dimensional torus $\TT^d$. Given $f\in \Leb^d_\#(\TT^d)$, we seek a uniformly bounded solution of the problem
\begin{equation}\label{eq:div}
\div \x = f, \qquad \x\in \Leb^\infty(\TT^d,\RR^d).
\end{equation}
The classical  elliptic solution of the first half of (\ref{eq:div}), $\x = \nabla\Delta^{-1}f$, lies in $W^{1,d}_\#(\TT^d)$ which   may fail to satisfy the  uniform bound sought in the second half of (\ref{eq:div}).
Thus, the question is whether (\ref{eq:div}) admits a solution which gains  uniform boundedness, $\|\x\|_{L^\infty} \lesssim \|f\|_{\Leb^d}$, at the expense of giving up on the  irrotationality condition $\curl \x = 0$. This question was addressed by Bourgain and Brezis, \cite[Proposition 1]{BB03}. They prove that \eqref{eq:div} admits uniformly bounded solutions for all $f \in \Leb^d_\#(\TT^d)$, with the intricate aspect that a solution operator mapping $\Leb^d_\#\mapsto \Leb^\infty$ \emph{must be nonlinear}; in particular, therefore, the uniform boundedness of irrotational elliptic solution \emph{must} fail.
 The existence  of such uniformly bounded solutions  was proved in \cite{BB03}  using a straightforward duality argument based on the closed range theorem.

 The purpose of this paper is to present an alternative approach for the existence of such solutions. Our approach is \emph{constructive}: 
the solution $\x$ is constructed as a \emph{hierarchical} sum, $\x=\sum_{j=1}^\infty\u_j$, where the $\{\u_j\}$'s are computed recursively as appropriate minimizers,
\[
  \u_{j+1} = \arginf{\u} \Big\{\|\u\|_{\Leb^\infty} + \lami 2^j\big\|f-\div\Big(\sum_{k=1}^j\u_k+\u\Big)\|^d_{\Leb^d}\Big\}, \quad j=0,1,\ldots,
\]
and $\lami$ is a sufficiently large parameter specified below. As an example, we refer  to \cite{TT11} for the computation of uniformly bounded hierarchical solution of the equation  $\div \x= \Delta G$ with $G:=x_1|\ln r|^{1/3}\zeta(r) \in \Leb^2_\#(\TT^2)$
where $\zeta(\cdot)$ is  a radial cut-off away from the origin, \cite[\S3, Remark 7]{BB03}. The elliptic solution, $\x=\nabla G$, has a fractional logarithmic growth at the  origin, whereas the computation confirms that the hierarchical solution $\x=\sum\u_j$ remains uniformly bounded, ${\|\x\|_{\Leb^\infty} \lesssim \|\Delta G\|_{\Leb^2} < \infty}$.

\medskip\noindent
The above construction  is in fact a special case of our main result which applies to 
 general linear problems of the form 
\begin{equation}\label{eq:int_Lxfp}
\T\x=f, \quad f\in\Leb^p_\#(\Om), \quad \Om\subset \RR^d, \ 1<p<\infty.
\end{equation}
Here, $\T:\B\mapsto \Leb^p_\#(\Om)$ is a linear operator densely  defined on a Banach space $\B$ with a closed range in  $\Leb^p_\#(\Om)$. The subscript $\{\cdot\}_\#$ indicates an  appropriate subspace of $\Leb^p$,
\[
\Leb^p_\#(\Om)= \Leb^p(\Om)\cap \mbox{Ker}(\P),
\]
 where $\P:\Leb^p\mapsto \Leb^p$ is a linear operator whose null is ``compatible" with the range of $\T$ so that the dual of $\T$ is injective, namely, there exists $\BVbeta>0$ such that
\begin{equation}\label{ass:int_Hp}
\|\rg-\Pdual\rg\|_{\Leb^{p'}} \leq \BVbeta\dualnorm{\Tdual\rg}, \qquad \forall \rg\in \Leb^{p'}(\Om).
\end{equation}
The closed range theorem combined with the open mapping principle tell us  that  if the {\it a priori} duality estimate assumed in  (\ref{ass:int_Hp}) holds, then equation \eqref{eq:int_Lxfp} admits a  solution, $\Bnorm{\x}\lesssim \BVgamma \|f\|_{\Leb^p}$ with a constant $\BVgamma=\BVgamma(\BVbeta,p,d)$.  Our main result explains the existence of such $\x$s by \emph{explicit construction}. 

\begin{thm}\label{thm:int_Lx=f_in_Lp}
Fix $1<p<\infty$ and assume that the {\it a priori} estimate \eqref{ass:int_Hp} holds.
There exists $\BVgamma<\infty$ (depending on $p$ and linearly on $\BVbeta$) such that for any given $f\in \Leb^p_\#(\Om)$, the equation 
$\T\x=f$  admits a  hierarchical solution of the form $\x=\sum_{j=\i}^\infty \u_j\in \B$, 
\begin{equation}\label{eq:int_x_in_B}
\Bnorm{\x}\leq \BVgamma \|f\|_{\Leb^p}, \qquad \BVgamma <\infty.
\end{equation} 
Here, the $\{\u_j\}$'s are constructed recursively as minimizers of 
\begin{equation}\label{eq:int_hie}
\u_{j+1}=\arginf{\u} \left\{\Bnorm{\u} + \lam_{j+1}\big\|f-\T\Big(\sum_{k=\i}^j\u_k +\u\Big)\big\|^p_{\Leb^p}\right\}, \quad  \  j=0,1,\ldots,
\end{equation}
where $\{\lam_{j}\}_{j\geq1}$ is a dyadic sequence, $\lam_{j+1}:=\lami \two^{j}$  with a sufficiency large $\displaystyle  \lami > \frac{\BVbeta}{p}\|f\|^{1-p}_{\Leb^p}$.
\end{thm}

\begin{rem}[Exponential convergence]\label{rem:sharp}
The description of $\x$ as the sum $\x=\sum\u_j$ provides a multiscale \emph{hierarchical decomposition} of a solution for \eqref{eq:int_Lxfp} for  rapidly increasing sequence of scales, $\lam_{j+1}=\lami \zeta^{j}$ with any $\zeta >1$. The role $\{\lam_j\}$'s as the different scales associated with the $\u_j$'s is reflected through the exponential decay bound (consult  \eqref{eq:finite} below)
\[
\Bnorm{\u_j} \lesssim \frac{\lam_{j+1}}{\lam^{p'}_j} \sim \|f\|_{\Leb^p_\#} \zeta^{-\frac{j}{p-1}}, \qquad \lam_{j+1}=\lam_1\zeta^j \sim 
\frac{\zeta^j}{\|f\|^{p-1}_{\Leb^p_\#}}, \qquad 1<p<\infty.
\]
For simplicity, we limit our discussion to the dyadic case $\zeta=2$. 
\end{rem}

\begin{rem}[On the {\it a priori} duality estimate \eqref{ass:int_Hp}]\label{rem:beta}
The {\it a priori} estimate \eqref{ass:int_Hp} is exactly what is needed for the hierarchical construction $\sum \u_j$ to converge.  It should be emphasized, however, that  the construction  does \emph{not} require   knowledge of the precise value of the constant $\BVbeta$ appearing in estimate \eqref{ass:int_Hp}. Indeed, the parameter $\BVbeta$ enters through the initial scale, $\lami$, which is to be chosen  large enough,
\[
\lami \geq \frac{\BVbeta}{p}\|f\|^{1-p}_{\Leb^p},
\]
so that by lemma \ref{cor:caseii}, it dictates a non-trivial first  hierarchical step, 
\[
\u_1=\arginf{\u} \left\{\Bnorm{\u} + \lami\big\|f-\T\u\big\|^p_{\Leb^p}\right\}.
\]
What happens if the initial scale $\lami$ is underestimated relative to an unknown value of $\BVbeta$? then, as  noted in lemma \ref{cor:casei} below,  the variational statement (\ref{eq:int_hie})   will yield zero hierarchical terms, $\u_j\equiv0$ for increasing sequence of scales $\lami 2^j,\ j=1,2,\ldots$, until reaching the critical scale such that $p\lami \two^{j_0} \geqsim \BVbeta\|f\|^{1-p}_{\Leb^p}$, which will dictate the first non-trivial step of the hierarchical decomposition, $\x=\sum_{j=j_0}\u_j$. 
In this sense, the  construction of hierarchical solution, $\x=\sum \u_j$ is \emph{independent} of  the precise value of $\BVbeta$ in  \eqref{ass:int_Hp}: the latter is only needed to guarantee that the hierarchical construction will indeed pick up the first  non-trivial minimizer after finitely many steps $j_0$. 
\end{rem}

\begin{rem}[The limiting cases $p=1,\infty$]
The $\Leb^p$-valued hierarchical constructions in theorem \ref{thm:int_Lx=f_in_Lp} can be extended to a more general setup of operators valued in Orlicz spaces (outlined in remark \ref{rem:orlicz} below). 
The  limiting cases, however,  are excluded; for example, there exist no $\Wdot^{1,p}$ solutions of $\div\x=f$ for general
$f\in\Leb^p$ with $p=1,\infty$, \cite[Section 2]{BB03}, \cite{DFT05}.  The iterative aspect of the hierarchical construction is reminiscent of Artola \& Tartar construction of  $\Leb^p(\RR)$-functions for the end case $p=1$, as a limiting case for interpolation of $W^{1,1}(\RR^2)$-traces, \cite[\S II]{Tar94},\cite{Ga57}.
\end{rem}
\noindent

$\Leb^2$-based hierarchical decompositions were introduced by us in the context of image processing, \cite{TNV04,TNV08}, and motivated the present  construction of solutions in the more general setup of the closed range theorem. We demonstrate
such hierarchical constructions of solution to two important examples of critical regularity  studied by Bourgain \& Brezis, \cite{BB03,BB07}. These are the construction of uniformly bounded solutions to $\div\x=f \in \Leb^d_\#(\TT^d)$,  discussed in section \ref{sec:divU=F}, and construction of uniformly bounded solutions to $\curl\x=\ff \in \Leb^3_\#(\TT^3,\RR^3)$ discussed in section \ref{sec:curlU=F}.
The critical regularity in these cases manifests itself in terms of lack of  right inverses for $\T$ bounded on the corresponding critical $\Leb^p$ spaces, or equivalently, $Ker{\T}$ which cannot be complemented in $\Leb^\infty$,
\cite[\S 3]{BB03},\cite[\S5.3]{Aji10}.

The main novelty of theorem \ref{thm:int_Lx=f_in_Lp} is  using these hierarchical decompositions for explicit construction of solutions  for general equations governed by operators with a closed range in $L^p_\#, \ 1<p<\infty$.  The proof of the special case $p=2$  is given in section \ref{subsec:peq2}: here, we trace precise bounds and clarify their role in the hierarchical construction. The $\Leb^2$-case    serves as the prototype case for the general setup of hierarchical constructions in  $\Leb^p$ spaces  in section \ref{sec:Lp}. Finally, the  characterization of minimizers, such as those encountered in \eqref{eq:int_hie}, is summarized in section \ref{sec:JTP}. 

\medskip\noindent
{\bf Acknowledgment}. I indebted to Haim Brezis for discussions on the works \cite{BB03,BB07}, to Fran\c{c}ois Golse, Giuseppe Savar\'{e} and Terence Tao who, respectively, brought to my attention \cite{Jam47}, \cite{Ga57,Tar94} and \cite{Aji10}, and to Przemyslaw Wojtaszczyk for his comments on an earlier version of this paper \cite{Woj10}.

\section{Bounded solutions of $\div \x=f\in \Leb^p_\#(\Omega,\RR)$}\label{sec:divU=F}
\addtocontents{toc}{\protect\setcounter{tocdepth}{1}}

Let $\P$ denote the averaging projection, $\P\rg:=\overline{\rg}$ where $\overline{\rg}$ is the average value of $\rg$. Given  $f\in \Leb^p_\#(\TT^d):=\big\{\rg\in\Leb^p(\TT^d)\ | \ \overline{\rg}=0\big\}$, then according to theorem \ref{thm:int_Lx=f_in_Lp},  we can construct hierarchical solutions of 
\begin{equation}\label{eq:divx=f}
\div \x=f, \qquad f\in \Leb^p_\#(\TT^d), \quad 1<p<\infty,
\end{equation}
in an appropriate Banach space, $\x\in\B$, provided the corresponding {\it a priori} estimate \eqref{ass:int_Hp} holds, namely, there exists a constant $\BVbeta>0$ (which may vary of course, depending on $p, d$ and $\B$), such that
\begin{equation}\label{ass:Hdivp}
\|\rg-\overline{\rg}\|_{\Leb^{p'}} \leq \BVbeta \dualnorm{\nabla\rg}, \qquad \forall \rg\in \Leb^{p'}(\TT^d).
\end{equation}
We specify four cases of such relevant $\B$'s.

\medskip\noindent
{\bf \#1. Solution  of} $\div\x=f\in \Leb^p_\#$ {\bf with}  $\x\in\Wdot^{1,p}$.
Since
\[
\|\rg-\overline{\rg}\|_{\Leb^{p'}(\TT^d)} \leq \|\nabla\rg\|_{\Wdot^{-1,p'}(\TT^d,\RR^d)}, \quad \forall \rg\in \Leb^{p'}(\TT^d),
\]
we can construct hierarchical solutions of (\ref{eq:divx=f}) in $\B=\Wdot^{1,p}(\TT^d,\RR^d), 1<p<\infty$. 
This is the same integrability space of the  irrotational solution of (\ref{eq:divx=f}), 
$\nabla\Delta^{-1}f\in\Wdot^{1,p}(\TT^d,\RR^d)$.

\medskip\noindent
{\bf \#2. Solution of} $\div\x=f\in \Leb^p_\#$ {\bf with} $\x\in\Leb^{p^*}$.
By Sobolev inequality 
\begin{equation}\label{eq:Sob}
\|\rg-\overline{\rg}\|_{\Leb^{p'}(\TT^d)} \leq \BVbeta \|\nabla\rg\|_{\Leb^{(p^*)'}(\TT^d,\RR^d)}, \qquad  \frac{1}{p^*}= \frac{1}{p}-\frac{1}{d}, \qquad d\leq p<\infty, \quad \forall \rg\in \Leb^{p'}(\TT^d), 
\end{equation}
where the case $p=d$  corresponds to the  isoperimetric Gagliardo-Nirenberg inequality, \cite{DPD02,CNV04} 
$\|\rg-\overline{\rg}\|_{\Leb^{d'}(\TT^d)} \leq \BVbeta \| \rg\|_{\BV(\TT^d)}$. We distinguish between two cases.

\smallskip
(i) The case $d<p<\infty$: the equation 
$\div \x=f \in \Leb^p_\#(\TT^d)$
 has a solution 
$\x\in \Leb^{p^*}(\TT^d,\RR^d)$. This is the same integrability space of the irrotational solution
 $\nabla\Delta^{-1}f \in W^{1,p}(\TT^d,\RR^d)\subset L^{p^*}(\TT^d,\RR^d)$. 

\smallskip
(ii) The case  $d=p$: the equation  $\div \x=f \in \Leb^d_\#(\TT^d)$
 has a solution $\x\in \Leb^\infty(\TT^d,\RR^d)$. This is the   the prototype example discussed in the beginning of the introduction. According to the  intriguing observation of  Bourgain \& Brezis, \cite[Proposition 2]{BB03}, there exists no bounded right inverse $K:\Leb^d_\# \mapsto \Leb^\infty$ for the operator $\div$, and therefore, there exists no \emph{linear} construction of solutions $f\mapsto \x$ (in particular, $\nabla\Delta^{-1}f$ cannot be uniformly bounded). Theorem \ref{thm:int_Lx=f_in_Lp} provides a \emph{nonlinear}  hierarchical construction of such solutions. The computation of such $\Leb^\infty$-solutions  using hierarchical iterations  in the two-dimensional critical case was carried out in \cite{TT11}.

\begin{rem}[Homogeneity]
 We rewrite the hierarchical iteration \eqref{eq:int_hie} with $\lami=C\|f\|^{1-p}_{\Leb^p}$ in the form
\[
\bra\u_{j+1},\r_{j+1}\ket=\arginf{\T\u+\r=\r_j} \left\{\Bnorm{\u} + C\two^j\frac{\big\|\r\big\|^p_{\Leb^p}}{\|f\|^{p-1}_{\Leb^p}}\right\}, \quad 
\r_j:=\left\{\begin{array}{ll} f, & j=0\\
f-\T\big(\sum_{k=\i}^j\u_k\big), & j=1,2\ldots .\end{array}\right.
\]
Observe that if $[\u_\i,\r_\i]$ is the first minimizer associated with $\r_0=f$, then  the corresponding first minimizer associated with  $\alpha f$ is $[\alpha\u_\i,\alpha\r_\i]$, and recursively, the next hierarchical components are $[\alpha\u_j,\alpha\r_j]$. Thus, as  noted in \cite[remark 1.1]{TNV08}, the hierarchical solution is homogeneous of degree one: if  $\x=\x_{f}$ specifies the (nonlinear) dependence of  hierarchical solution on $f$, then $\x_{\{\alpha f\}}=\alpha\x_{f}$.
\end{rem}

\medskip\noindent
{\bf \#3. Solution  of} $\div\x=f\in \Leb^d_\#$ {\bf with} $\x\in\Leb^{\infty}\cap \Wdot^{1,d}$. 
A central question raised and answered in \cite{BB03} is whether (\ref{eq:divx=f}) has a solution which captures the \emph{joint} 
regularity, $\x\in \B=\Leb^{\infty}\cap \Wdot^{1,d}(\TT^d,\RR^d)$. To this end, one needs to verify the duality estimate  \eqref{ass:Hdivp}, which now reads
\begin{equation}\label{eq:Wdot_Lpq}
\|\rg-\overline{\rg}\|_{\Leb^{d'}(\TT^d)} \leq \BVbeta \|\nabla \rg\|_{\Leb^1 + \Wdot^{-1,d'}(\TT^d,\RR^d)}, \qquad \forall \rg\in \Leb^{d'}(\TT^d). 
\end{equation}
This  key estimate was proved in \cite{BB03}. Thus,  theorem \ref{thm:int_Lx=f_in_Lp}  converts (\ref{eq:Wdot_Lpq}) into  a constructive proof of:

\begin{cor}\label{cor:divU=f} The equation $\div\x=f\in \Leb^d_\#(\TT^d)$ admits a solution 
 $\x\in \Leb^{\infty}\cap \Wdot^{1,d}(\TT^d,\RR^d)$,  given by the hierarchical decomposition $\x=\sum_{j=1}\u_j$, which is constructed by the refinement step,
\[
\u_{j+1}=\arginf{\u} 
\left\{\|\u\|_{\Leb^{\infty}\cap \Wdot^{1,d}}+\lami \two^j\big\|f-\div\Big(\sum_{k=1}^j\u_k +\u\Big)\big\|^d_{\Leb^d}\right\}, \quad j=0,1,2\ldots,
\]
with sufficiently large $\lami \geq \frac{\BVbeta}{d}\|f\|^{1-d}_{\Leb^d}$. 
\end{cor}

\begin{rem}
We comment here the key role of the duality estimate \eqref{eq:Wdot_Lpq}. The case $d=2$  was proved by a direct method outlined in \cite[Section 4]{BB03}; alternative two-dimensional proofs were given by  Maz'ya \cite{Ma07} and Mironescu \cite{Mi10}.  For $d>2$, however,  the derivation of \eqref{eq:Wdot_Lpq} was proved in Bourgain \& Brezis \cite[theorem 1]{BB03} as a byproduct of their \emph{construction} of $\Leb^{\infty}\cap \Wdot^{1,d}$ solutions for $\div\x=f$!.  The construction,  based on an intricate Littlewood-Paley decomposition  is rather involved \cite[section 6]{BB03},  and to  our knowledge, a simpler, \emph{direct} derivation of \eqref{eq:Wdot_Lpq} is still open.
Thus, corollary \ref{cor:divU=f}  --- which still depends on  the construction of Bourgain \& Brezis  to justify \eqref{eq:Wdot_Lpq}, offers a simpler alternative for the  construction of such  $\Leb^\infty\cap\Wdot^{1,d}$-bounded solutions in terms of the minimizers,  
\[
\J_{\div}(\r,\lam):=\mbox{inf}_{\u}\big\{\|\u\|_{\Leb^\infty\cap\Wdot^{1,d}}+\lam \|r-\div\u\|_{\Leb^d}^d\big\}.
\]
Computation of the related  $\Leb^\infty$-based minimizers were carried out in \cite{GLMV07,LV05} and it would be desirable to develop efficient algorithms to compute the corresponding minimizers of $\J_{\div}(\r,\lam;\Leb^\infty\cap\Wdot^{1,d})$. Spectral approximation of such minimizers was discussed in \cite{Ma06}.

\smallskip\noindent
Since the proof of the dual estimate  \eqref{eq:Wdot_Lpq} in $d>2$ dimensions is indirect, a specific value of $\BVbeta$ is not known. As noted in remark \ref{rem:beta}, however, the hierarchical construction can proceed  without {\it a priori} knowledge of the precise value of $\BVbeta$:  if one  sets $\lami=\|f\|^{1-d}_{\Leb^d}$ and this  initial scale  underestimates a correct value of, say,  $\BVbeta>1$, then it will take at most $j_0\sim \log(\BVbeta)$  steps  before picking-up non-trivial terms in the  hierarchical decomposition, $\x=\sum_{j=j_0}\u_j$. 
\end{rem}

\medskip\noindent
{\bf \#4. Solution  of} $\div\x=f\in \Leb^d_\#(\Om)$ {\bf with} $\x\in\Leb^{\infty}\cap \Wdot^{1,d}_0(\Om)$.
 The constructions of bounded solutions for (\ref{eq:divx=f}) extend to the case of Lipschitz domains,  $\Omega\subset \RR^d$, see \cite[section 7.2]{BB03}. For future reference we state the following.

\begin{cor}\label{cor:divU=finOm} Given $f\in \Leb^d_\#(\Om):=\big\{g\in\Leb^d(\Om)\ | \ \displaystyle \int_\Om g(x) dx=0\big\}$, then  the equation ${\div\x=f}$ admits a solution 
 $\x\in \Leb^{\infty}\cap \Wdot^{1,d}_0(\Om,\RR^d)$,  such that
\[
\|\x\|_{\Leb^{\infty}\cap \Wdot^{1,d}(\Om)} \leq \BVgamma \|f\|_{\Leb^d(\Om)}.
\]
It is given by the hierarchical decomposition, $\x=\sum_{j=1}\u_j$, 
which is constructed by the refinement step,
\[
\u_{j+1}=\arginf{\u: \ \u_{|\partial\Om}=0} 
\left\{\|\u\|_{\Leb^{\infty}\cap \Wdot^{1,d}(\Om)}+\lami\two^j\big\|f-\div\Big(\sum_{k=1}^j\u_k+\u\Big)\big\|^d_{\Leb^d(\Om)}\right\}, \ j=0,1,2\ldots,
\]
with sufficiently large $\lami \geqsim {\BVbeta}{\|f\|^{1-d}_{\Leb^d(\Om)}}$.
\end{cor}

\section{Bounded solution of $\curl \x=\ff\in \Leb^3_\#(\TT^3,\RR^3)$}\label{sec:curlU=F}
Let $\Leb^3_\#(\TT^3,\RR^3)$ denote the $\Leb^3$-subspace of  divergence-free 3-vectors with zero mean. We seek solutions of
\begin{equation}\label{eq:curlx=f}
\curl \x=\ff, \qquad \ff\in \Leb^3_\#(\TT^3,\RR^3),
\end{equation}
in an appropriate Banach space, $\x\in\B$. We appeal to the framework of hierarchical solutions in theorem \ref{thm:int_Lx=f_in_Lp}, where  $\P:\Leb^3(\TT^3,\RR^3)\mapsto \Leb^3(\TT^3,\RR^3)$ is the irrotational portion of Hodge decomposition with a dual, 
 $\Pdual\bg:=\nabla\Delta^{-1}\div\bg-\overline{\bg}$.
According to theorem \ref{thm:int_Lx=f_in_Lp},  we can construct hierarchical solutions, $\x\in \B$ of (\ref{eq:curlx=f}),
 provided (\ref{ass:int_Hp}) holds
\begin{equation}\label{ass:Hcurlp}
\|\bg-\Pdual\bg\|_{\Leb^{3/2}} \leq \BVbeta \dualnorm{\curl\bg}, \qquad \bg\in \Leb^{3/2}(\TT^3,\RR^3).
\end{equation}

Since $\|\bg-\Pdual\bg\|_{\Leb^{3/2}} \lesssim \|\curl\bg\|_{\Wdot^{-1,3/2}}$, 
we can construct hierarchical solutions of (\ref{eq:curlx=f}) in $\Wdot^{1,3}$. 
This has the same integrability as the  divergence-free solution of (\ref{eq:curlx=f}), 
$(-\Delta)^{-1}\curl \ff$. A more intricate question is whether (\ref{eq:curlx=f}) admits a uniformly bounded solution, since such a solution \emph{cannot} be constructed by a linear procedure. These solutions were constructed by Bourgain and Brezis in \cite[Corollary 8']{BB07}, 
which in turn imply the key {\it a priori} estimate,
\begin{equation}\label{eq:key_curl}
\|\bg-\Pdual\bg\|_{\Leb^{3/2}(\TT^3,\RR^3)} \leq \BVbeta \|\curl\bg\|_{L^1+\Wdot^{-1,3/2}(\TT^3,\RR^3)},
\qquad \forall \bg\in \Leb^{3/2}(\TT^3,\RR^3): \ \div\bg=\overline{\bg}=0. 
\end{equation}
Granted \eqref{eq:key_curl}, theorem \ref{thm:int_Lx=f_in_Lp} offers a simpler alternative to the construction in \cite{BB07} based on the following hierarchical decomposition.
\begin{cor}\label{cor:curlX=f}
The equation $\curl\x=\ff\in \Leb^3_\#(\TT^3,\RR^3)$,  admits a solution $\x\in\Leb^\infty \cap \Wdot^{1,3}(\TT^3,\RR^3)$,
\[
\|\x\|_{\Leb^\infty \cap \Wdot^{1,3}(\TT^3,\RR^3)} \leq  \BVgamma\|\ff\|_{\Leb^3(\TT^3,\RR^3)},
\]
which can be constructed by the (nonlinear) hierarchical expansion, $\x=\sum\u_j$,
\[
\u_{j+1}=\arginf{\u} \left\{\|\u\|_{\Leb^\infty \cap \Wdot^{1,3}}+\lami\two^j\big\|\ff-\curl\Big(\sum_{k=1}^j\u_k+\u\Big)\big\|^3_{\Leb^3(\TT^3,\RR^3)}\right\}, \quad j=0,1,\ldots,
\]
with sufficiently large $\lami\geq \frac{\BVbeta}{3}\|\ff\|^{-2}_{\Leb^3(\TT^3,\RR^3)}$.
\end{cor}

\section{Construction of hierarchical solutions for  $\T \x=f \in \Leb^p_\#(\Om)$}\label{sec:Lp}
\addtocontents{toc}{\protect\setcounter{tocdepth}{2}}
\subsection{A prototype example: hierarchical solution of $\div \x=f\in \Leb^2_\#(\TT^2)$}\label{subsec:peq2}

We begin our discussion on hierarchical constrictions with a two-dimensional prototype example of
\begin{equation}\label{eq:div2}
\div \x = f, \qquad f\in \Leb^2_\#(\TT^2):=\big\{\rg\in \Leb^2(\TT^2)\ \big| \ \int_{\TT^2}g(x) dx=0\big\}.
\end{equation}
Our starting point for the construction of a uniformly bounded solution of (\ref{eq:div2})
is a  decomposition of $f$, 
\begin{subequations}\label{eqs:Qa}
\begin{equation}\label{eq:diva}
f = \div\u_\i+\r_\i, \qquad f\in \Leb^2_\#(\TT^2),
\end{equation}
where $\bra \u_\i,\r_\i \ket$ is a 
minimizing pair of the functional, 
\begin{equation}\label{eq:Qdiva}
\bra \u_\i,\r_\i\ket = \arginf{\div\u+\r=f} \Big\{\|\u\|_{\Leb^\infty} + \lami \|\r\|^2_{\Leb^2}\Big\}.
\end{equation}
\end{subequations}
Here $\lami$ is a fixed parameter at our disposal:  if we choose $\lami$ large enough, 
$\displaystyle \lami >\frac{1}{2\|f\|_{\BV}}$,
then  according to lemma \ref{cor:caseii} below, (\ref{eq:Qdiva}) admits a 
minimizer, $\bra \u_\i,\r_\i \ket$, satisfying, 
\[
\|\r_\i\|_{\BV} =\frac{1}{2\lam_\i}.
\]
To proceed we invoke the isoperimetric Gagliardo-Nirenberg inequality, which states that there exists $\BVbeta>0$ (any $\BVbeta\geq 1/\sqrt{4\pi}$ will do), such that for all bounded variation $g$'s with zero mean,
\begin{equation}\label{eq:BV}
\|g\|_{\Leb^2} \leq\BVbeta \|g\|_{\BV}, \qquad \int_{\TT^2}g(x)dx=0
\end{equation}
Since $f$ has a zero mean so does the residual $\r_\i$ and (\ref{eq:BV}) yields
\[
\|\r_\i\|_{\Leb^2}  \leq \BVbeta\|\r_\i\|_{\BV} =\frac{\BVbeta}{2\lam_\i}.
\]
We conclude that the residual $\r_\i \in \Leb^2_\#(\TT^2)$, and we can therefore implement  the same variational decomposition  of $f$ in  (\ref{eqs:Qa}), and use it to decompose $\r_\i$ with  scale $\displaystyle \lam=\lamii>\lami=\frac{1}{2\|\r_\i\|_{\BV}}$. This  yields
\[
\r_\i=\div\u_\ii+\r_\ii, \qquad \bra \u_\ii,\r_\ii\ket = \arginf{\div\u+\r=\r_\i} \Big\{\|\u\|_{\Leb^\infty} + \lamii \|\r\|^2_{\Leb^2}\Big\}.
\]

Combining this with (\ref{eq:diva}) we obtain 
$f=\div \x_\ii +\r_\ii$, where $\x_\ii:=\u_\i+\u_\ii$  is viewed as an improved \emph{approximate solution} of (\ref{eq:div2}) in the sense that it has a smaller  residual, $\r_\ii$, 
\[
\|\r_\ii\|_{\Leb^2} \leq \BVbeta\|\r_\ii\|_{\BV} =\frac{\BVbeta}{2\lamii}.
\] 
when compared with the previous residual $\BVbeta\|\r_\i\|_{BV} =\frac{\BVbeta}{2\lami}$.
This process can be repeated: if $\r_j\in \Leb^2_\#(\TT^2)$ is the residual at step $j$, then we decompose it

\begin{subequations}\label{eqs:Qj}
\begin{equation}\label{eq:Qja}
\r_j=\div\u_{j+1}+\r_{j+1},
\end{equation}
where $\bra\u_{j+1},\r_{j+1}\ket$ is a minimizing pair of
\begin{equation}\label{eq:Qjb}
\bra \u_{j+1},\r_{j+1}\ket = \arginf{\div\u+\r=\r_j} \Big\{\|\u\|_{\Leb^\infty} + \lam_{j+1} \|\r\|^2_{\Leb^2}\Big\}, \qquad j=0,1,\ldots. 
\end{equation}
\end{subequations}
For $j=0$, the decomposition (\ref{eqs:Qj}) is interpreted as (\ref{eq:diva}) by setting $\r_0:=f$.
Note that the recursive decomposition (\ref{eq:Qja}) depends on the invariance of 
$\r_j\in \Leb^2_\#(\TT^2)$: if $\r_j$ has a zero mean then so does $\r_{j+1}$, and by (\ref{eq:BV}) $\r_{j+1}\in \Leb^2_\#(\TT^2)$.  
The iterative process depends on a sequence of increasing scales, 
$ \lami < \lamii < \ldots \lam_{j+1}$, which are yet to be determined. 

The telescoping sum of the first $k$ steps in (\ref{eq:Qja}) yields an improved approximate solution, 
$\x_k:=\sum_{j=\i}^k \u_j$:
\begin{equation}\label{eq:ites}
f = \div \x_k +\r_k, \qquad \|\r_k\|_{\Leb^2} \leq \BVbeta \|\r_k\|_{\BV} = \frac{\BVbeta}{2\lam_k}\downarrow 0, \qquad k=1,2,\ldots.
\end{equation}
The  key question is  whether the $\x_k$'s remain uniformly bounded, and it is here that we use the freedom in choosing the scaling parameters $\lam_k$: comparing the minimizing pair $[\u_{j+1},\r_{j+1}]$ of (\ref{eq:Qjb}) with the trivial pair $[\u\equiv 0, \r_j]$ implies, in view of (\ref{eq:ites}),
\begin{equation}\label{eq:iter}
\|\u_{j+1}\|_{\Leb^\infty}+ \lam_{j+1}\|\r_{j+1}\|_{\Leb^2}^2   \leq \lam_{j+1} \|\r_j\|_{\Leb^2}^2 \leq \left\{\begin{array}{ll} \lami\|f\|_{\Leb^2}^2, \quad & j=0, \\ \\
{\displaystyle \frac{\BVbeta^2\lam_{j+1}}{4\lam_j^{2}}}, \quad & j=1,2,\ldots . \end{array}\right.
\end{equation}

We conclude that by choosing a sufficiently fast increasing $\lam_j$'s such that $\sum_j \lam_{j+1}\lam^{-2}_j< \infty$, then the approximate solutions $\x_k=\sum^k_1 \u_j$ form a Cauchy sequence in $\Leb^\infty$ whose limit, $\x=\sum^\infty_1 \u_j$, satisfies the following.

\begin{thm}\label{thm:main2}
Fix $\BVbeta$ such that \eqref{eq:BV} holds. Then, for any given $f\in \Leb^2_\#(\TT^2)$, there exists a uniformly bounded solution of \eqref{eq:div2},
\begin{equation}\label{eq:main}
\div \x= f, \qquad \|\x\|_{\Leb^\infty} \leq 2\BVbeta \|f\|_{\Leb^2}.
\end{equation} 
The solution $\x$ is given by $\x=\sum_{j=\i}^\infty \u_j$, where the $\{\u_j\}$'s are constructed recursively as  minimizers of 
\begin{equation}\label{eq:hira}
 \bra \u_{j+1},\r_{j+1}\ket = \arginf{\div\u+\r=\r_j} \Big\{\|\u\|_{\Leb^\infty} + \lami \two^j \|\r\|^2_{\Leb^2}\Big\}, \qquad \r_0:=f, \quad \lami=\frac{\BVbeta}{\|f\|_{\Leb^2}}.
\end{equation}
\end{thm}

\begin{proof}
With $\lam_j=\lami \two^{j-1}$ we have $\|\x_k-\x_\ell\|_{\Leb^\infty} \lesssim \sum \lam_{j+1}\lam_j^{-2} \lesssim \two^{-\ell}, \ k>\ell\gg1$. Let $\x$ be the  limit of the Cauchy sequence $\{\x_k\}$ then $ \|\x_j - \x\|_{\Leb^\infty}+\|\div\x_j-f\|_{\Leb^2}\lesssim \two^{-j} \rightarrow 0$, and since $\div$ has a closed graph on its domain ${D}:=\{\u\in \Leb^\infty: \ \div\u\in \Leb^2(\TT^2)\}$, it follows that $\div\x=f$. By (\ref{eq:iter}) we have
\[
\|\x\|_{\Leb^\infty} \leq \sum_{j=\i}^\infty \|\u_j\|_{\Leb^\infty}
\leq \lami\|f\|_{\Leb^2}^2 + \frac{\BVbeta^2}{4\lami}\sum_{j=\ii}^\infty\frac{1}{\two^{j-3}} =\lami \|f\|_{\Leb^2}^2 +  \frac{\BVbeta^2}{\lam_\i}. 
\]
Here $\displaystyle \lami> \frac{1}{2\|f\|_{\BV}}$ is a free parameter at our disposal:  we choose   $\lami := \BVbeta/ \|f\|_{\Leb^2}$ which by (\ref{eq:BV})  is admissible, 
$\displaystyle \lami = \frac{\BVbeta}{\|f\|_{\Leb^2}}  > \frac{1}{2\|f\|_{\BV}}$,
and  \eqref{eq:main} follows.
\end{proof}

\begin{rem}[Energy decomposition]\label{re:energy} A telescoping summation of the left inequality of (\ref{eq:iter}) yields 
\[
\sum_{j=\i}^\infty \frac{1}{\lam_j}\|\u_j\|_{\Leb^\infty} \leq \|f\|_{\Leb^2}^2;
\]
setting $\displaystyle \lam_j=\frac{\beta \two^{j-1}}{2\|f\|_{\Leb^2}} $ we conclude the ``energy bound"
\begin{equation}\label{eq:energy2}
\sum_{j=\i}^\infty \frac{1}{\two^{j-1}}\|\u_j\|_{\Leb^\infty}\leq\frac{\beta}{2}\|f\|_{\Leb^2}.
\end{equation}
In fact, a precise energy \emph{equality} can be formulated in this case, using the characterization of the minimizing pair (consult theorem \ref{thm:thm1} below), $2(\r_{j+1},\div\u_{j+1})=\|\u_{j+1}\|_{\Leb^\infty}/\lam_{j+1}$: by squaring the refinement step
$\r_{j}=\r_{j+1}+\div\u_{j+1}$  we find
\[
\|\r_j\|^2_{\Leb^2} - \|\r_{j+1}\|^2_{\Leb^2} = 2 (\r_{j+1},\div\u_{j+1}) + \|\div\u_{j+1}\|^2_{\Leb^2}=  \frac{1}{\lam_{j+1}}\|\u_{j+1}\|_{\Leb^\infty} + \|\div\u_{j+1}\|^2_{\Leb^2}.
\]
A telescoping sum of the last  equality yields
\begin{cor}\label{cor:L2energy}
Let $\x=\sum_{j=1}^\infty \u_j\in \Leb^\infty$ be a hierarchical solution of $\div\x=f, \ f\in\Leb^2_\#(\TT^2)$. Then 
\begin{equation}\label{eq:div_L2energy}
\frac{1}{\lami}\sum_{j=\i}^\infty \frac{1}{\two^{j-1}}\|\u_{j}\|_{\Leb^\infty} + \sum_{j=\i}^{\infty}\|\div\u_{j}\|^2_{\Leb^2_\#(\TT^2)} = \|f\|^2_{\Leb^2_\#(\TT^2)}, \qquad \lam_1=\frac{\beta}{2\|f\|_{\Leb^2}} 
\end{equation}
\end{cor}
\end{rem}

\noindent
We mention two examples related to the two-dimensional setup of theorem \ref{thm:main2}.
 
\subsubsection{Oscillations and image processing}
As noted earlier, there exists no linear construction of solutions of  \eqref{eq:div2} for \emph{general} $f\in \Leb^2$. Yet,  for the `slightly smaller' Lorenz space, $\Leb^{2,1}$, we have 
\[
\nabla\Delta^{-1}f\in \Leb^\infty, \qquad f\in\Leb^{2,1}_\#(\TT^2).
\] 
(we note in passing that  $\Leb^{2,1}$ is a limiting case for the linearity of $f\mapsto \x$ to survive  the $\Leb^{2,\infty}$-based nonlinearity result argued in the proof of \cite[proposition 2]{BB03}). 
Thus, the nonlinear aspect of constructing hierarchical solutions for \eqref{eq:div2}
becomes essential for highly oscillatory functions such that $f\in\Leb^2\backslash\Leb^{2,1}$ (and in particular, $f\notin BV(\TT^2)$). Such $f$'s are encountered in image processing in the form of noise, texture, and blurry images, \cite{Me02,BCM05}. 
Hierarchical decompositions in this context of images  were introduced by us in \cite{TNV04} and were found to be  effective tools in image de-noising, image de-blurring and image registration, \cite{BCM05,LPSX06,PL07,TNV08,HRC10,AXRNW13,TH13}, including graph-based signals \cite{HLTE10,HLE13}. Here, we are given a  noisy and possibly blurry observed image, $f=\T\x+\r \in \Leb^2(\RR^2)$,  and the purpose is to recover a faithful description of the underlying `clean' image,  $\x \sim ``\T^{-1}" f$, by de-noising $\r$ and de-blurring $\T$. The inverse $``\T^{-1}" f$ should  be properly interpreted, say,  in the smaller space $BV(\RR^2)\subset \Leb^2(\RR^2)$ which is known to be well-adapted to represent edges.  The resulting inverse problem can solved by corresponding variational problem of \cite{ROF92,CL95,CL97} 
\begin{equation}\label{eq:tik}
\bra \u,\r\ket=\arginf{\T \u+\r=f} \big\{\|\u\|_{BV} + \lam \|\r\|_{\Leb^2(\RR^2)}^2\big\},
\end{equation}
which is a special case of \emph{Tikhonov-regularization}, \cite{TA77,Mo84,Mo93}.  
The   $(BV,\Leb^2)$-hierarchical decomposition corresponding to (\ref{eq:tik})
reads, \cite{TNV04,TNV08},
\begin{equation}\label{eq:BVL2}
f\eqsim \T\x_m, \quad \x_m=\sum_{j=1}^m \u_j, \qquad \bra \u_{j+1},\r_{j+1}\ket=\arginf{\T\u+\r=\r_j}\left\{\|\u\|_{BV}+\lami 2^j\|\r\|^2_{L^2}\right\}.
\end{equation} 

The oscillatory nature of noise and texture in images was addressed by  Y. Meyer who advocated,  \cite{Me02}, to replace 
$\Leb^2$ with the larger space of ``images" $G:=\{f \ | \div \u=f, \u\in \Leb^\infty\}$. The equation $\div\u=f$  arises here with \emph{one-signed} measure $f$'s, and its $\Leb^\infty$ solutions were characterized in \cite[\S1.14]{Me02},\cite{PT08}: the space 
$G_+$ coincides with Morrey space $M_+^2(\Om)$:
\[
M^2(\Om)=\big\{ \mu\in {\mathcal M} \  \big| \int_{B_r}d\mu \lesssim r, \quad \forall B_r \subset \Om\big\}.
\]
For one-signed measure, $M_+^2(\Om)$ coincides with Besov space $\dot{B}^{-1,\infty}_{\infty}$. The corresponding 
Meyer's energy functional then reads, $\displaystyle \bra \u,\r\ket=\arginf{\T \u+\r=f} \big\{\|\u\|_{BV(\Om)} + \lam \|\r\|_{\dot{B}^{-1,\infty}_{\infty}(\Om)}\big\}$; numerical simulations with this model are found in \cite{VO04}.

\subsubsection{$\Leb^1(\TT^2)$-bounds and $\Hdot^{-1}(\TT^2)$-compactness}
Here is a simple application  of theorem \ref{thm:main2}.
Let $f\in \Hdot^{-1}(\TT^2)$ be given. For arbitrary $\rg\in \Hdot^{1}(\TT^2)$  we have
$\xi_j\widehat{\rg}(\xi)\in \Leb^2_\#(\TT^2)$, and  by theorem \ref{thm:main2},
there exist  bounded $\x_{ij}\in\Leb^\infty(\TT^2)$, such that

\begin{displaymath}
\begin{array}{ll}
\left\{
\begin{array}{ccc}
\xi_1\widehat{\rg}(\xi) &= &\xi_1\widehat{\x}_{11}(\xi) + \xi_2\widehat{\x}_{12}(\xi),\\
\xi_2\widehat{\rg}(\xi) &= &\xi_1\widehat{\x}_{21}(\xi) + \xi_2\widehat{\x}_{22}(\xi),
\end{array}
\right. & \quad \|\x_{ij}\|_{\Leb^\infty} \lesssim \|\rg\|_{\Hdot^{1}(\TT^2)}.
\end{array}
\end{displaymath}
Thus, expressed in terms of the Riesz transforms, $\widehat{R_j\psi}(\xi):=\widehat{\psi}(\xi)\xi_j/|\xi|$, we have
\[
\rg=\frac{1}{2}\left(\x_{11}+\x_{22}\right) +\frac{1}{2}\left(R_1^2-R_2^2\right)\left(\x_{11}-\x_{22}\right) + R_1R_2\left(\x_{12}+\x_{21}\right);
\] 
Since $R_1^2-R_2^2$ and $R_1R_2$ agree up to rotation, we conclude that: every $\rg\in \Hdot^1(\TT^2)$ can be written as the sum 
\[
 \rg=\x_1 + R_1R_2\x_2, \quad \|\x_1\|_{\Leb^\infty}+\|\x_2\|_{\Leb^\infty} \lesssim  \|\rg\|_{\Hdot^{1}(\TT^2)} \quad
\text{for \ all \ } \rg\in \Hdot^1(\TT^2).
\]
Here, $U_1,U_2$ are given by linear combination of the $U_{ij}$'s in their Cartesian and their rotated coordinates. The  last representation shows that although an $\Leb^1(\TT^2)$-bound of $f$ does not imply $f\in \Hdot^{-1}(\TT^2)$, then $f$ does belong to $\Hdot^{-1}$ if $f$ \emph{and} its repeated Riesz transform, $R_1R_2f$, are $L^1$-bounded. 
\begin{cor}\label{cor:Hmone}
The following bound holds
\begin{equation}\label{eq:tmp345}
\|f\|_{\Hdot^{-1}(\TT^2)} \lesssim \|f\|_{L^1(\TT^2)} + \|R_1R_2f\|_{L^1(\TT^2)}.
\end{equation}
\end{cor}
As an example, consider a family of divergence-free 2-vector fields, 
$\ueps(t,\cdot)\in \Leb^2(\TT^2,\RR^2)$, which are approximate solutions of two-dimensional incompressible Euler's equations. 
One is interested in their convergence to a proper weak solution, with no concentration effects, \cite{DM87}. It was shown in \cite{LNT00} that $\{\ueps\}$ converges to such a weak solution  if the vorticity, $\omeps(t\cdot):=\partial_1 \uepsb(t,\cdot)-\partial_2\uepsa(t,\cdot)$, is compactly embedded in $H^{-1}(\TT^2)$. By corollary \ref{cor:Hmone}, $H^{-1}$-compactness holds if $\{R_1R_2\omeps(t,\cdot)\} \hookrightarrow \Leb^1(\TT^2)$; consult \cite{Ve92}.

\subsection{Hierarchical solutions for  $\T \x=f \in \Leb^p_\#(\Om)$: approximate solutions}

We turn our attention to the construction of hierarchical solutions for equations of the general form 
 \begin{equation}\label{eq:Lxf}
\T \x= f, \qquad f\in \OPhid(\Om), \qquad 1<p<\infty.
\end{equation}
A solution $\x$ is sought in a Banach space $\B:=\{ \x\,:\, \Bnorm{\x}< \infty \}$.
The general framework,  involving two linear operators, $\T$ and $\P$, is outlined below.

\medskip
The linear operator $\T$ is densely defined on  $\B$ with a closed range in $\OPhid := \OPhi \cap \mbox{Ker}(\P)$ with appropriate $\P:\OPhi\mapsto \OPhi$. We let $\Tdual:\OPsi \mapsto \Bdual$ denote the formal dual of $\T$, acting on $\OPsi$  
with the natural pairing (effectively, $\Tdual$ is acting on $\OPsid:=\OPsi\cap\mbox{Ker}(\P)$, since $R(\Pdual)$ is in the null of $\Tdual$)
\[
\langle \Tdual\rg,\u\rangle =(\rg,\T\u), \qquad \rg\in\OPsi, \ \u\in\B,
\]
and let $\dualnorm{\cdot}$ denote the dual norm 
\[
\dualnorm{\Tdual\rg}:=\sup_{\u\neq 0}\frac{\langle\Tdual\rg,\u\rangle}{\Bnorm{\u}}, \qquad \rg\in \OPsi.
\]

We begin by constructing an \emph{approximate solution} of (\ref{eq:Lxf}), $\x_{\lam}: \T \x_{\lam} \approx f$, such that the residual 
$\r_{\lam}:=f-\T \x_{\lam}$  is driven to be small by a proper choice of a \emph{scaling parameter} $\lam$  at our disposal. The approximate solution is obtained in terms of minimizers of the variational problem,
\begin{equation}\label{eq:blur}
 \JT(f,\lam):= \inf_{\T\u+\r=f} 
\Big\{ \Bnorm{\u}+ \lambda\Lp{\r} \,: \, \u\in \B, \ \r\in \OPhid\Big\}.
\end{equation}
In theorem \ref{thm:thm1} below we show if $\lambda$ is chosen  sufficiently large, 
\begin{equation}\label{eq:const}
\lam  >\frac{1}{p\dualnorm{\Lin{f}}}, \qquad \myphi{f}:=\mbox{sgn}(f)|f|^{p-1},
\end{equation}
then the functional $\JT(f,\lam)$ in (\ref{eq:blur}) admits a 
minimizer, $\u=\u_{\lam}$,  such that the size of the residual, $\r_{\lam}:=f-\T \u_{\lam}$, is dictated by the dual statement
\begin{equation}\label{eq:dual}
\dualnorm{\Lin{\r_\lam}}=\frac{1}{p\lam }.
\end{equation}

Fix the scale $\lam=\lami >1/p\dualnorm{\Lin{f}}$. We  construct an approximate solution, $\T \x_{\i} \approx f, \ \x_{\i}:=\u_{\i}$, where $\u_{\i}$ is  a 
minimizer of $\JT(f, \lami)$,
\[
f=\T \u_{\i}+\r_{\i}, \qquad \bra \u_{\i},\r_{\i}\ket=\arginf{\T \u+\r=f} \JT(f, \lami)
\]
Borrowing the terminology from image processing we note that the corresponding residual $\r_{\i}$ contains `small' features which were left out of $\u_{\i}$.  Of course,  whatever is interpreted as `small' features  at a given $\lami$-scale, may contain  significant 
features when viewed under a refined scale, say $\lamii>\lami$. To this end we \emph{assume} that the residual $\r_{\i}\in \OPhid$
so that we can repeat the $\JT$-decomposition of $\r_{\i}$, this time  at the refined scale $\lamii$:
\[
\r_{\i}=\T \u_{\ii}+\r_{\ii}, \qquad 
\bra \u_{\ii},\r_{\ii}\ket=\arginf{\T \u+\r=\r_{\i}} \JT(\r_{\i}, \lamii). 
\]

Combining the last two steps 
we arrive at a better two-scale representation of $\x$ given by 
$\x_\ii := \u_{\i}+\u_{\ii}$, as an improved approximate  solution of $\T \x_\ii \approx f$.
Features below scale $\lamii$ remain unresolved in $\x_\ii$, but the process  can be continued  by successive application of the  refinement step,
\begin{equation}\label{eq:hrb}
\r_j = \T \u_{j+1}+\r_{j+1}, \quad \bra \u_{j+1},\r_{j+1}\ket:= \arginf{\T \u+\r=\r_j} \JT(\r_j, \lam_{j+1}),  \quad j=\i,\ii,... .
\end{equation}
To enable this process we require   the residuals $\r_j$ to remain in $\OPhid$. In view of the dual bound (\ref{eq:dual}), we therefore make the following

\noindent
{\bf Assumption} ({A closure bound}){\bf .} 
{\it There exists  a constant} $\BVbetap=\BVbetap(p,d) <\infty$ {\it such  that the following a priori estimate holds}
\begin{equation}\label{ass:closure}
  \Lp{\rg} \leq  \BVbetap \dualnorm{\Tdual\myphi{\rg}}^{p'}, \qquad \myphi{\rg}=\mbox{sgn}(\rg)|\rg|^{p-1}.
\end{equation}
We postpone the discussion of this bound to theorem \ref{thm:L2-data} below and we continue with the generic  hierarchical step where   $\bra \u_{j+1},\r_{j+1}\ket$ is constructed as a 
minimizing pair of $\JT(\r_j,\lam_{j+1})$: since this minimizer is characterized by satisfying $\dualnorm{\Tdual\myphi{\r_{j+1}}}=1/p\lam_{j+1}$, then the closure bound  \eqref{ass:closure} implies that   $\r_{j+1}\in \Leb^p$; moreover, since $\r_j$ and $R(\T)$ are in $\mbox{Ker}(\P)$ then, 
\[
\r_{j+1}= r_j-\T\u_{j+1} \in \mbox{Ker}(\P),
\]
and we conclude that $\r_{j+1}\in \OPhid$. In this manner, the iteration step 
$\bra \u_j,\r_j\ket \mapsto \bra\u_{j+1},\r_{j+1}\ket$, 
is well-defined on $\B\times \OPhid$.
After $k$  such steps we have, 
\begin{eqnarray}
f&=&\T\u_{\i}+\r_{\i}= \nonumber \\
&=&\T\u_{\i}+\T\u_{\ii}+\r_{\ii} =\nonumber \\
&=& \ldots  \ldots \qquad \qquad \ \ \ = \nonumber \\
&=&\T\u_{\i}+\T\u_{\ii}+ \dots +\T\u_{k}+\r_{k}. \nonumber
\end{eqnarray}

We end up with a multiscale \emph{hierarchical representation} of an approximate solution of (\ref{eq:Lxf}) $\x_k:=\sum_{j=\i}^k \u_j\in \B$ such that $\T \x_k \eqsim f$.
Here, the approximate equality $\eqsim$  is interpreted as the  convergence of the residuals,
\[
\dualnorm{\Lin{\r_k}} = \frac{1}{\lambda_k} \rightarrow 0, \qquad \r_k:=f-\T\x_k,
\]
 dictated by the sequence of scales,
$\lami<\lamii <\ldots<\lam_k$, which is at our disposal.
We summarize with the following theorem.
\begin{thm}[Approximate solutions]\label{thm:L2-data}
Consider $\T:\B\mapsto \OPhid(\Om)$ and assume its dual is injective so that \eqref{ass:int_Hp} holds,
\[
\|\rg-\Pdual\rg\|_{\Leb^{p'}} \leq \BVbeta\dualnorm{\Tdual\rg}, \qquad \forall \rg\in \Leb^{p'}(\Om),
\]
for some $\P:\OPhi\mapsto \OPhi$ whose range is "compatible" with the range of $\T$.
Then, the equation $\T \x=f \in \OPhid(\Omega)$ admits an \emph{approximate solution}, 
$\x_k\in \B$, such that $\T \x_k \eqsim  f$ in the sense that the residuals $\r_k:=f-\T\x_k$ satisfy
\begin{equation}\label{eq:hierarchicala}
\dualnorm{\Lin{\r_{k}}} =  \frac{1}{p\lam_{k}}, \qquad
\qquad \r_k:=f-\T\x_k.
\end{equation}
The approximate solution  admits the  hierarchical expansion, $\x_k=\sum_{j=\i}^k \u_j$, where
the $\u_j$'s are constructed as minimizers,
\[
\bra \u_{j+1},\r_{j+1}\ket =\arginf{\T\u+\r=\r_j}\big\{\Bnorm{\u}+ \lam_{j+1}\Lp{\r}\big\}, \qquad \r_0=f.
\]
\end{thm}

\begin{proof} We  verify that the {\it a priori} duality estimate assumed in \eqref{ass:int_Hp} implies the closure bound sought in  (\ref{ass:closure}).  
Fix $\rg\in \Leb^{p}_\#(\Om)$. Then $\myphi{\rg}:={\rm sgn}(\rg)|\rg|^{p-1}\in \Leb^{p'}(\Om)$ and since  $\rg\in \mbox{Ker}(\P)$ we find
\[
\int_{\Om}  |\rg|^{p}dx =\int_{\Om}\rg \myphi{\rg}dx = \int_{\Om}\rg \big(\myphi{\rg}-\Pdual\myphi{\rg}\big) dx   \leq  \|\rg\|_{L^p}\|\myphi{\rg}-\Pdual\myphi{\rg}\|_{L^{p'}}.
\]
The {\it a priori} dual estimate assumed in  \eqref{ass:int_Hp}  then yields
\[
\|\rg\|^{p}_{\Leb^p} \leq \|\rg\|_{\Leb^p}\|\myphi{\rg}-\Pdual\myphi{\rg}\|_{\Leb^{p'}} \leq \BVbeta\|\rg\|_{\Leb^p}\dualnorm{\Tdual \myphi{\rg}}, \qquad \forall \rg\in \Leb^p_\#(\Om),
\]
and the closure bound \eqref{ass:closure} follows with $\BVbetap:=\BVbeta^{p'}$,
\begin{equation}\label{eq:closure}
\|\rg\|_{\Leb^p}^{p-1} \leq \BVbeta\dualnorm{\Tdual\myphi{\rg}}.
\end{equation}
This allows us to proceed with  the hierarchical iterations \eqref{eq:hrb}, 
\[
\bra \u_j,\r_j \ket \in \B\times \Leb^p_\# \ \ \ \mapsto \ \ \ \bra \u_{j+1},\r_{j+1}\ket:= \arginf{\T \u+\r=\r_j} \JT(\r_j, \lam_{j+1}) \in \B\times \Leb^p_\#, \quad j=1,2, \ldots
\]
starting with $\bra \u_0,\r_0\ket =\bra 0,f\ket$.
A telescoping summation of (\ref{eq:hrb}) yields an approximate solution $\x_k= \sum_{j=1}^k \u_j$
such that its residual $\r_k=f-\T\x_k$ satisfies  \eqref{eq:hierarchicala}.
\end{proof}

\begin{rem}[on the closure bound]
As an example for the closure bound \eqref{eq:closure} for $\T$'s with an injective dual, consider the critical case of $\T=\div: \Leb^\infty \mapsto \Leb^d(\TT^d)$ and let $\P$ denote the zero averaging projection $\P\rg=\rg-\overline{\rg}$. The corresponding dual estimate \eqref{ass:int_Hp} reads
\[
\|g-\overline{\rg}\|_{\Leb^{d'}} \lesssim \|\Tdual\rg\|_{\BV}.
\]
This is the isoperimetric Gagliardo-Nirenberg inequality  and it implies, along the lines of  theorem \ref{thm:L2-data}, the following closure bound corresponding to (\ref{ass:closure})
\[
\|\rg\|^{d-1}_{\Leb^d(\TT^d)} \lesssim \|\mbox{sgn}(\rg)|\rg|^{d-1}\|_{\BV(\TT^d)}, \quad \forall \rg\in \Leb^d_\#(\TT^d).
\]
Equivalently, we can rewrite this as $\|\myphi{\rg}\|_{\Leb^{d'}} \lesssim \|\myphi{\rg}\|_{\BV(\TT^d)}$. The observant reader will notice that the latter  is a slight variant of Gagliardo-Nirenberg inequality  since for $d>2$, $\myphi{\rg}$ need not have zero average; only $\rg$ does.
\end{rem}

\subsection{From approximate to exact solutions}
We turn to show that the approximate solutions, $\x_k=\sum^k_{j=1} \u_j$, converge to a limit $\x=\sum^\infty_{j=1} \u_j$, which is an exact solution sought for (\ref{eq:Lxf}), uniformly bounded in $\B$. 

We start by comparing the minimizer $[\u_{j+1},\r_{j+1}]$ of $\JT(\r_j,\lam_{j+1})$ in \eqref{eq:hrb} with the trivial pair  ${[\u\equiv0,\r_j]}$, which yields the key refinement estimate 
\begin{equation}\label{eq:basicidentity}
 \Lp{\r_{j}}\geq   \frac{1}{\lam_{j+1}}\Bnorm{\u_{j+1}} +  \Lp{\r_{j+1}}, \qquad j=0,1,\ldots.
\end{equation}
In particular, the closure bound (\ref{ass:closure})  followed by (\ref{eq:hierarchicala}) imply
\begin{equation}\label{eq:finite}
\Bnorm{\u_{j+1}} \leq \lam_{j+1}\Lp{\r_j} 
  \left\{\begin{array}{ll}
\displaystyle = \lami\Lp{f}, & j=0, \\ \\
\displaystyle \leq \lam_{j+1}\BVbetap \dualnorm{\Tdual\myphi{\r_j}}^{1/p'} \leq \frac{\lam_{j+1}\BVbetap}{(p\lam_j)^{p'}}, & j=1,2,\ldots,\end{array}\right.
\end{equation} 
where $\{\lam_j\}$ is an increasing sequence of  scales at our disposal.
Setting $\lam_j=\lami \two^{j-1}$, we conclude that the approximate solutions, $\x_k=\sum^k_1 \u_j$ form a Cauchy sequence, 
\[
\Bnorm{\x_k-\x_\ell}  \lesssim  \sum_{j=\ell+1}^k \two^{j(1-p')}\ll 1, \qquad k > \ell\gg 1,
\]
which has a limit,  $\x=\sum_{j=\i}^{\infty} \u_j$, such that  $\Lp{\T\x_j-f}\rightarrow 0$. Since $\T$ has a closed range, $\T\x=f$. 
It remains to show that the limit $\x$ has a finite $\B$-norm, which brings us to 

\noindent
\begin{proof}[{\bf Proof of theorem} \ref{thm:int_Lx=f_in_Lp}] Using (\ref{eq:finite}) with $\BVbetap=\beta^{p'}$ yields 
\[
\Bnorm{\x} \leq \Bnorm{\u_1} + \sum_{j=\i}^\infty \Bnorm{\u_{j+1}}  
  \leq  \lami \Lp{f}+ \sum_{j=\i}^\infty  \frac{\lami \two^{j}\BVbetap}{\left(p\lami \two^{j-1}\right)^{p'}} \leq
\lami \|f\|^p_{\Leb^p} + \left(\frac{2\BVbeta}{p}\right)^{p'}\frac{1}{\lami^{\!\!\!\!\!\!p'-1}(\two^{\!p'-1}-1)}.
\]

\noindent
Set  $\lami :=  \frac{2\BVbeta}{p}\|f\|^{1-p}_{\Leb^p}$. Such a choice of $\lami$  satisfies the admissibility requirement \eqref{eq:const}: indeed,  according to \eqref{ass:closure}, $\|g\|^{p-1}_{\Leb^p} \leq \BVbeta \dualnorm{\Tdual\myphi{g}}$, hence
\[
\lami = \frac{2\BVbeta}{p}\|f\|^{1-p}_{\Leb^p}  > \frac{1}{p\dualnorm{\Tdual\myphi{f}}},
\]
and the uniform bound \eqref{eq:int_x_in_B} follows, 
\begin{equation}\label{eq:tmp123}
\Bnorm{\x}\leq \BVgamma\|f\|_{\Leb^p_\#}, \qquad \BVgamma= \frac{2\BVbeta}{p}\left(1+\frac{1}{\two^{p'-1}-1}\right).
\end{equation}
\end{proof}

\begin{rem}
We summarize the two main aspects in the hierarchical construction.

(i) The existence  minimizers $\{\u_j\}$ of $\JT(\r_{j-1}, \lam_j)$, which follows from basic principles in uniformly convex Banach spaces.
We  use here the mere existence of such minimizers, $\{\u_j\}$,  instead of standard duality-based existence arguments in the closed range theorem, e.g., \cite[VII.5]{Yos80},\cite[I.A.13-14]{Woj91},\cite[theorem 2.20]{Bre10}. We note in passing that  existence of minimizers and duality principles in  uniformly convex Banach spaces can be deduced from each other,  \cite{Jam47}.

(ii) The  exponential decay of these minimizers and hence the uniform bound of their sum, 
$\Bnorm{\x}\leq \sum \Bnorm{\u_j} \lesssim \|f\|_{\Leb^p_\#}$,  follow from the key {\it a priori} dual estimate \eqref{ass:int_Hp} used in   the refinement step \eqref{eq:finite}.
\end{rem}

\begin{rem}[Extension to Orlicz spaces]\label{rem:orlicz} The hierarchical construction extends to  equations valued in more general Orlicz spaces,
\begin{equation}\label{eq:orlicz}
\T\x=f \in \Leb^\Phi_\#:=\Leb^\Phi\cap \text{Ker}(\Pdual), \qquad \Leb^\Phi=\{f \ : \ [f]_\Phi:=\int_\Om \Phi(|f|)dx < \infty \},
\end{equation}
for a proper $N$-function $\Phi$, satisfying the $\Delta_2$ condition, \cite[\S8]{AF03},\cite[\S4.8]{BS88}. Assume that the  following {\it a priori} closure bound holds: there exists an increasing function $\eta: \RR_+ \mapsto \RR_+$ such that 
\[
[\rg]_\Phi \leq \eta(\Bnorm{\Tdual\myphi{\rg}}), \qquad \int_{s=0}^1 \frac{\eta(s)}{s^2}ds <\infty.
\]
Then, the problem \eqref{eq:orlicz} admits the bounded hierarchical solution, $\x=\sum\u_j$, such that $\Bnorm{\x} \lesssim [f]_\Phi$. The closure bound enters through the initial scale $\lami \geqsim 1/\eta^{-1}([f]_\Phi)$. The $\Leb^p$ setup corresponds to $\Phi(t)=t^p$ and $\eta(s) \sim s^{p'}$. 
\end{rem}

\begin{rem}[Sharp bounds]
The bound \eqref{eq:tmp123} with $p=2$ shows that if $\Tdual$ is injective so that \eqref{ass:int_Hp} holds with constant $\BVbeta$, then   $\T\x=f\in \Leb^2$  admits a solution   $\Bnorm{\x}\leq \BVgamma\|f\|_{\Leb^2}$, with twice the bound  $\BVgamma =2\BVbeta$ (in agreement with the $L^2$-case in theorem  (\ref{thm:main2})). Using a rapidly growing scales, $\lam_{j+1}=\lami \zeta^j$ with $\zeta\gg 1$ yields a tighter bound $\gamma$. A sharp form of the $\B$-bound (\ref{eq:tmp123}) for general $1< p< \infty$,
\begin{equation}{\label{eq:sharp}}_{\hspace*{-5cm} \gamma} \hspace*{3.9cm}
\Bnorm{\x} \leq \gamma \|f\|_{\Leb^p_\#} \quad \text{for \emph{any}} \quad \gamma>\beta,
\end{equation}
can be argued by Hahn-Banach theorem. To this end, we reproduce here a slight generalization of \cite[proposition 1]{BB03}.
Normalize  $\|f\|_{\Leb^p}=1$ and consider the two non-empty convex sets: the ball 
\[
B_{\gamma_\epsilon}:=\{\u\in\B: \ \Bnorm{\u} < \gamma_\epsilon\}, \quad \gamma_\epsilon:=(1+\epsilon)\beta,
\]
and $C:=\{\x\in \B: \ \T\x=f\}$.
The claim is that $B_{\gamma_\epsilon}\cap C\neq \emptyset$ and  the desired estimate \eqref{eq:sharp}${}_{{\gamma_\epsilon}}$ then follows with arbitrarily small $\epsilon$.
If not, $B_{\gamma_\epsilon}\cap C=\emptyset$, and by Hahn-Banach there exists a non-trivial $\rg^*\in \Leb^{p'}$ such that
for some $\alpha \in \mathbb{R}_+$
\begin{subequations}
\begin{equation}\label{eq:sub}
\langle \rg^*,\u\rangle \leq \alpha, \quad \forall \u\in B_{\gamma_\epsilon}
\end{equation}
and
\begin{equation}\label{eq:sup}
\langle \rg^*,\x\rangle \geq \alpha, \quad \forall \x\in C.
\end{equation}
\end{subequations}
If ${V}\in \mbox{Ker}(\T)$ then application of (\ref{eq:sup}) with $\x \mapsto \x\pm \delta {V} \in C$ yields
$\pm\delta \langle \rg^*,{V}\rangle \geq 0$, or $\langle \rg^*,{V}\rangle = 0$; that is, $\rg^*\in \mbox{Ker}(\T)^\perp=\mbox{R}(\T^*)$ is of the form $\rg^*=\Tdual\rg$ for some $\rg\in D(\Tdual)\subset \Leb^{p'}$.
Now,  by  (\ref{eq:sub}) 
\[
\dualnorm{\rg^*} = \sup_{\Bnorm{\u}=\gamma_{\epsilon/2}}\frac{\langle\rg^*,\u\rangle}{\gamma_{\epsilon/2}}\leq \frac{\alpha}{\gamma_{\epsilon/2}},
\]
and the {\it a priori} estimate assumed in  \eqref{ass:int_Hp} implies 
\[
\|\rg\|_{\Leb^{p'}_\#} \leq \beta\dualnorm{\T^*\rg} = \beta\dualnorm{\rg^*}\leq \frac{\alpha}{1+\epsilon/2}.
\]
But this leads to a contradiction: pick $\x\in C$ (which we recall is not empty) then (\ref{eq:sup}) implies,
\[
\alpha \leq \langle \rg^*,\x\rangle =\langle \Tdual\rg,\x\rangle =\langle \rg, f\rangle \leq \|\rg\|_{\Leb^{p'}}\|f\|_{\Leb^p}
\leq \frac{\alpha}{1+\epsilon/2}. \qedhere
\]
\end{rem}

\section{An appendix on $\JT$-minimizers}\label{sec:JTP} 

To study the hierarchical expansions (\ref{eq:hrb}), we characterize the minimizers of the $\JT$-functionals (\ref{eq:blur})
\begin{equation}\label{eq:JTP}
 \bra \u,\r\ket:= \arginf{\T \u+\r=f} \JT(f,\lam) , \qquad \JT(f,\lam):= \inf_{\T\u+\r=f}\Big\{ \Bnorm{\u}+ \lambda\Lp{\r} \,: \, \u\in \B\Big\}.
\end{equation}
Here $\T:\B\mapsto \OPhid(\Omega)$ is densely defined into a subspace of $\Leb^p(\Omega)$ over a Lipschitz domain $\Omega \subset \RR^d$.
 The characterization summarized below extends related results which can be found in 
\cite[Theorem 4]{Me02}, \cite[Chapter1]{ACM04}, \cite[Theorem 2.3]{TNV08}.

Recall that $\dualnorm{\cdot}$ denotes the \emph{dual} norm, 
$\dualnorm{\Tdual\rg}=\mathop{\sup}_{\u}\langle \Tdual\rg,\u\rangle/\Bnorm{\u}$, so that the usual duality bound holds
\begin{equation}\label{eq:Pdual}
\langle \Tdual\rg,\u\rangle\leq \Bnorm{\u}\dualnorm{\Tdual\rg}, \qquad \rg\in\OPsi, \ \u\in \B.
\end{equation} 
We say that $\u$ and $\Tdual\rg$ is  an \emph{extremal pair} if equality holds above.
The theorem below characterizes $\bra\u,\r\ket$ as  a minimizer of the $\JT$-functional if and only if $\u$ and $\Tdual\myphi{\r}$ form an extremal pair.

\begin{thm}\label{thm:thm1}
Let $\T:\B\rightarrow \OPhid(\Omega)$ be a linear 
operator with dual  $\Tdual$ and let 
$\JT(f,\lam)$ denote the associated functional \eqref{eq:blur}.  

(i) The variational problem  \eqref{eq:JTP} admits a minimizer $\u$. Moreover, if $\Bnorm{\cdot}$ is \emph{strictly convex}, then the minimizer $\u$ is unique.

(ii) $\u\in \B$ is a minimizer of \eqref{eq:JTP}  if and only if the residual $\r:=f-\T\u$ satisfies
\begin{equation}\label{cond}  
\langle  \Tdual\myphi{\r}, \u\rangle=\Bnorm{\u}\cdot\dualnorm{\Tdual\myphi{\r}}=\frac{\Bnorm{\u}}{p\lambda}, 
\qquad  \myphi{r}:=\mbox{sgn}(\r)|\r|^{p-1}\in\OPsi.
\end{equation} 
\end{thm}

\begin{proof}

(i) The existence of a minimizer for the $\JT$-functional follows from  standard arguments which we omit, consult \cite{AV94,Me02}. We address the issue of uniqueness. Assume $\u_1$ and $\u_2$ are minimizers with the corresponding residuals $\r_1=f-\T\u_1$ an $\r_2=f-\T\u_2$
\[
\Bnorm{\u_i} + \lam\Lp{\r_i}  = \qmin, \quad i=1,2
\]
We then end up with the one-parameter family of minimizers, $\u_\theta:=\u_1 + \theta (\u_2-\u_1), \ \theta \in [0,1]$,
\[
\qmin \leq \Bnorm{\u_\theta} + \lam \Lp{\r_\theta} 
 \leq \theta \Bnorm{\u_2} + (1-\theta)\Bnorm{\u_1} + \theta \lam \Lp{\r_2}
+ (1-\theta) \lam \Lp{\r_1} = \qmin.
\]
Consequently,  $\Lp{\r_\theta}=\theta  \Lp{\r_2} + (1-\theta)  \Lp{\r_1}$ and hence $\r_1=\r_2$. In particular, $\Lp{\r_1}=\Lp{\r_2}$ implies that the  two minimizers satisfy $\Bnorm{\u_1}=\Bnorm{\u_2}$ and we conclude that the ball $\Bnorm{\u}=\Bnorm{\u_1}\neq 0$  contains the  segment $\{\u_\theta, \ \theta\in [0,1]\}$, which  
by strict convexity, must be the trivial segment, i.e., $\u_2=\u_1$.
\newline
We note in passing that strict convexity is in fact \emph{necessary} for uniqueness, e.g., the counterexample  of lack of uniqueness  over the $\ell^\infty$-unit ball,  \cite[pp. 40]{Me02}.  

(ii) If   $\u$ is a minimizer of (\ref{eq:JTP}) then 
 for any $\h\in \B$ we have
\begin{eqnarray}\label{eq:tmpa}
\lefteqn{ \ \ \JT(\u,\lam)= \Bnorm{\u} +\lambda\Lp{f-\T\u} } \\
 & &\leq  \JT(\u+\epsilon\h,\lam)=\Bnorm{\u+\epsilon \h} + \lambda\Lp{f-\T(\u+\epsilon \h)} \nonumber \\
&  & \leq   \Bnorm{\u}+|\epsilon|\cdot\Bnorm{\h} + \lambda\Lp{f-\T\u}-\lambda\epsilon p \Big( \mbox{sgn}(f-\T\u)|f-\T\u|^{p-1},\T\h\Big)  +o(\epsilon). \nonumber
\end{eqnarray}
It follows that for all $\h \in \B$,
\[
\Big| \big\langle \Tdual\myphi{\r},\h\big\rangle \Big| \leq \frac{1}{p\lambda}\Bnorm{\h} +o(1),
\qquad  \myphi{\r}=\mbox{sgn}(\r)|\r|^{p-1}, \ \r:=f-\T\u,
\]
and by letting $\epsilon \rightarrow 0$
\begin{equation}\label{eq:inqa}
\dualnorm{\Lin{\r}} \leq \frac{1}{p\lambda}.
\end{equation}

\noindent
To verify the reverse inequality, we set  $\h= \pm \u$ and $0<\epsilon<1$ in (\ref{eq:tmpa}), yielding
\[
\Bnorm{\u} +\lambda\Lp{f-\T \u}  \leq (1\pm \epsilon)\Bnorm{\u} + \lambda\Lp{f-\T \u \mp\epsilon \T \u},
\]
and hence $\pm \epsilon\Bnorm{\u}\mp p\lambda \epsilon \big( \myphi{f-\T\u},\T \u\big)  + o(\epsilon)\geq 0$.
Dividing by $\epsilon$ and letting $\epsilon\downarrow 0_+$, we obtain $\Bnorm{\u}={p\lambda}\langle \Tdual\myphi{\r},\u \rangle$ and (\ref{cond}) follows:
\[
\frac{1}{p\lambda} \Bnorm{\u}=\langle \Tdual\myphi{\r},\u \rangle \leq \dualnorm{\Lin{\r}}\Bnorm{\u} \leq \frac{1}{p\lambda}\Bnorm{\u}.
\]

Conversely, we show that if (\ref{cond}) holds then $\u$ is a minimizer. The convexity of $\Leb^p$ yields
\begin{eqnarray*}
\lefteqn{\Lp{f-\T(\u+\h)} = \Lp{r-\T\h} =} \\ 
& & \geq \|r\|^p_{\Leb^p(\Omega)} -p\big( \mbox{sgn}(\r)|\r|^{p-1}, \T (\u+\h)\big) 
+ p\big( sgn(\r)|\r|^{p-1}, \T \u\big) \\
& & =   \Lp{f-\T \u} 
- \ p\overbrace{\big\langle \Tdual\myphi{\r}, (\u+\h)\big\rangle}^{\#1} \  
+ \ p\overbrace{\big\langle \Tdual\myphi{\r}, \u \big\rangle}^{\#2}. 
\end{eqnarray*}
By the equalities assumed in (\ref{cond}),  $\dualnorm{\Lin{\r}}={1}/{p\lambda}$, which implies  
$-p\lambda (\#1)\geq -\Bnorm{\u+ \h}$; moreover, $p\lam(\#2)=\Bnorm{\u}$.
We conclude that for any $\h\in \B$,  
\begin{eqnarray*}
\JT(\u+\h, \lam)=\Bnorm{\u+\h}+\lambda\Lp{f-\T(\u+\h)}
& \geq & \Bnorm{\u+\h} + \lam\Lp{f-\T \u} - p\lam(\#1) + p\lam (\#2) \\
& \geq  & \Bnorm{\u}+\lam\Lp{f-\T \u}=\JT(\u,\lam).
\end{eqnarray*}
Thus, $\u$ is a minimizer of (\ref{eq:JTP}). 
\end{proof}

The next two assertions are a refinement of Theorem \ref{thm:thm1}, depending on the size of $\dualnorm{\Tdual\myphi{f}}$.

\begin{lem}[The case $\dualnorm{\Tdual\myphi{f}} \leq 1/p\lambda$]\label{cor:casei}
Let $\T:\B\rightarrow \OPhid(\Omega)$ with adjoint  $\Tdual$ and let $\JT$ denote the associated functional \eqref{eq:blur}.  Then $p\lam\dualnorm{\Tdual\myphi{f}}\leq 1$ if and only if $\u\equiv 0$ is a minimizer of \eqref{eq:JTP}.
\end{lem}

\begin{proof}
Assume $\dualnorm{\Lin{f}}\leq 1/p\lam$. Then by convexity of $\Leb^p$
\begin{eqnarray*}
\Bnorm{\u}+\lambda \Lp{f-\T \u}&\geq&
\Bnorm{\u}+\lambda\int_\Omega |f|^p dx - p\lambda \int_\Omega\big( \myphi{f}, \T\u \big) dx \\
&\geq & \Bnorm{\u}+\lambda\int_\Omega |f|^p dx - p\lambda\dualnorm{\Tdual\myphi{f}}\Bnorm{\u} \geq \lambda \Lp{f},
\end{eqnarray*}
which tells us that  $\u\equiv0$ is a minimizer of (\ref{eq:blur}). 
Conversely, if $\u\equiv 0$ is a minimizer of (\ref{eq:JTP}), then 
$\epsilon\Bnorm{\u} + \lam\Lp{f-\epsilon\T\u} \geq \Lp{f}$ for all $\u\in\B$. It follows that
\[
\epsilon\Bnorm{\u} -\epsilon p\lam\int_{\Omega}(\myphi{f},\T\u)dx +o(\epsilon) \geq 0.
\]
 Letting $\epsilon \downarrow 0$ we conclude   $p\lam\langle\Tdual\myphi{f},\u\rangle \leq \Bnorm{\u}$, hence $\dualnorm{\Tdual\myphi{f}}\leq 1/p\lambda$. 
\end{proof}

\begin{lem}[The case $\dualnorm{\Tdual\myphi{f}}> 1/p\lambda$]\label{cor:caseii}
Let $\T:\B\rightarrow \OPhid(\Omega)$  with adjoint  $\Tdual$ and let $\JT$ denote the associated functional \eqref{eq:blur}.  
If  $1 < p\lam\dualnorm{\Tdual\myphi{f}} <\infty$, 
then $\u$ is a minimizer of \eqref{eq:JTP}  if and only if $\T\u$ and $\myphi{\r}$ is an extremal pair and 
\begin{equation}\label{condii}  
 \dualnorm{\Tdual\myphi{\r}}=\frac{1}{p\lambda}, \qquad \langle\u,\Tdual\myphi{\r}\rangle = \frac{\Bnorm{\u}}{p\lambda}.
\end{equation} 
\end{lem}
\begin{proof} Since $\dualnorm{\Lin{f}}>1/p\lam$ then  $\Bnorm{\u}\neq 0$ and we can now divide the equality on the right of (\ref{cond}) by $\Bnorm{\u}\neq 0$ and (\ref{condii}) follows. 
\end{proof}


\end{document}